\documentclass[12pt]{amsart}
\usepackage{amssymb,amscd,amsmath}
\usepackage{verbatim} 

\usepackage{amsfonts}
\usepackage{fullpage}
\usepackage{array}
\usepackage{amsthm}
\usepackage{url}
\usepackage{mathrsfs}
\newtheorem{theorem}{Theorem}[section]

\newtheorem{lemma}[theorem]{Lemma}
\newtheorem{corollary}[theorem]{Corollary}

\newtheorem{conjecture}[theorem]{Conjecture}
\theoremstyle{definition}
\newtheorem{definition}[theorem]{Definition}
\newtheorem{remark}[theorem]{Remark}

\newtheorem*{remark*}{Remark}
\newtheorem*{proposition*}{Proposition}

\newcommand{\GL}{\text{GL}}
\newcommand{\Spec}{\text{Spec}}
\input xy
\xyoption{all}
\title{Automorphicity and Mean-Periodicity}
\author{Thomas Oliver}
\thanks{Heilbronn Institute for Mathematical Research, University of Bristol, Bristol, UK.}
\date{\today}
\begin{document}
\begin{abstract}
If $C$ is a smooth projective curve over a number field $k$, then, under fair hypotheses, its $L$-function admits meromorphic continuation and satisfies the anticipated functional equation if and only if a related function is $\mathfrak{X}$-mean-periodic for some appropriate functional space $\mathfrak{X}$. Building on the work of Masatoshi Suzuki for modular elliptic curves, we will explore the dual relationship of this result to the widely believed conjecture that such $L$-functions should be automorphic. More precisely, we will directly show the orthogonality of the matrix coefficients of $\GL_{2g}$-automorphic representations to the vector spaces $\mathcal{T}(h(\mathcal{S},\{k_i\},s))$, which are constructed from the Mellin transforms $f(\mathcal{S},\{k_i\},s)$ of certain products of arithmetic zeta functions $\zeta(\mathcal{S},2s)\prod_{i}\zeta(k_i,s)$, where $\mathcal{S}\rightarrow\Spec(\mathcal{O}_k)$ is any proper regular model of $C$ and $\{k_i\}$ is a finite set of finite extensions of $k$. To compare automorphicity and mean-periodicity, we use a technique emulating the Rankin-Selberg method, in which the function $h(\mathcal{S},\{k_i\},s))$ plays the role of an Eisenstein series, exploiting the spectral interpretation of the zeros of automorphic $L$-functions.
\end{abstract}
\maketitle
\tableofcontents
\section{Introduction}
One reason to be interested in the analytic properties of Hasse--Weil $L$-functions is that they hint at the stronger, and much sought-after, automorphicity of these objects. Indeed, starting with Weil, various ``converse theorems'' have been proved, providing analytic conditions under which Dirichlet series arise from automorphic representations. An important premise of such theorems is that of holomorphic continuation of the series to the complex plane $\mathbb{C}$. This may not be weakened to merely meromorphic continuation. For example one may divide two automorphic $L$-functions to obtain a new Dirichlet series which is not automorphic (for example, it is not in the Selberg class), though it does share many analytic properties, except, of course, holomorphy. 
\newline Quotients of $L$-functions arise naturally as the zeta functions of arithmetic schemes. For example, if $\mathcal{S}\rightarrow\Spec(\mathcal{O}_k)$ is a proper, regular model of a smooth projective curve $C$ over a number field $k$, then one has
\[\zeta(\mathcal{S},s)=n(\mathcal{S},s)\frac{\zeta(k,s)\zeta(k,s-1)}{L(C,s)},\]
where $\zeta(\mathcal{S},s)=\prod_{x\in\mathcal{S}_0}\frac{1}{1-|k(x)|^{-s}}$, the product being taken over the closed points of the arithmetic surface. Ignoring temporarily the term $n(\mathcal{S},s)$, the expression on the right hand side is an ``Euler characteristic of $L$-functions''. Each term in the quotient is a Hasse--Weil $L$-function, in which the Euler factor is the reciprocal characteristic polynomial of the action of the Frobenius operator on inertia invariants of \'{e}tale cohomology, i.e. \[L_{\mathfrak{p}}(H^i(C),s)=\det(1-f_{\mathfrak{p}}N\mathfrak{p}^{-s}|H^i(C)^{I_{\mathfrak{p}}})^{-1}.\] 
Full details for the definition of Hasse--Weil $L$-functions are given in \cite{FLDFZDVA}. The function $n(\mathcal{S},s)$ depends on the choice of model $\mathcal{S}$ of $C$. In fact, it is a finite product of functions rational in variables of the form $p^{-s}$, where $p$ ranges over the residual characteristics of bad reduction. An exercise in \'{e}tale cohomology shows that (the appropriate completion of) $n(\mathcal{S},s)$ satisfies the correct functional equation with respect to $s\mapsto2-s$. We thus see that $L(C,s)$ admits meromorphic continuation if and only if $\zeta(\mathcal{S},s)$ does, and their functional equations are equivalent. In this paper we will focus on the zeta functions of arithmetic surfaces.
\newline Within modern harmonic analysis, long established as an effective tool in the study of number theory, one finds the notion of mean-periodicity. First introduced in $1935$ \cite{LFMP}, the general theory of mean-periodic functions is exposed in \cite{LOMPF}. Recent research places mean-periodicity at the heart of some of the basic open questions concerning the analytic properties of zeta functions of arithmetic schemes\footnote{These ideas originated in Ivan Fesenko's higher dimensional adelic analysis and duality, though here we will not use any of this theory.} \cite{F4}, \cite{F3}, \cite{SRF}. In fact, under fair hypotheses, the meromorphic continuation and functional equation is equivalent to the mean-periodicity of a related object. In light of converse theorems, and their limitations, it is natural to ask what the relationship is between mean-periodicity and the expected automorphicity properties of the Hasse--Weil $L$-functions that appear as factors of zeta. This question was first studied by Suzuki for modular elliptic curves in \cite{TDAAACAROGL2}, and in this paper we generalize his results to arbitrary smooth projective curves over number fields. 
\newline In order to provide some intuition, we begin with some indication of what it means for a function $f$ to be mean-periodic. First, we stress that this is not intrinsic to $f$, rather, it depends upon the spectral properties of the topological functional space $\mathfrak{X}$ within which we consider $f$ to be an element. Let a locally compact topological group $G$ act continuously on $\mathfrak{X}$, it is possible to define the set of ``translates'' of $f\in\mathfrak{X}$ as
\[T(f):=\{g\cdot f:g\in G\}.\]
We say that $f$ is $\mathfrak{X}$-mean-periodic if $T(f)$ is not dense in $\mathfrak{X}$. For example, let $\mathfrak{X}$ be the set of smooth functions on $\mathbb{R}$ and let $G$ be the additive group $\mathbb{R}$, which acts on $\mathfrak{X}$ by $g\cdot f(x)=f(x-g)$. In this we see that we have generalised the definition of periodicity. When $\mathfrak{X}$ satisfies the Hahn-Banach theorem, the definition of mean-periodicity just given is equivalent to the existence of a non-trivial homogeneous convolution equation $f\ast f^{\ast}=0$ for some $f^{\ast}\in \mathfrak{X}^{\ast}$. For example, if a smooth function $f:\mathbb{R}\rightarrow\mathbb{C}$ has period $a\in\mathbb{R}$, then it satisfies the convolution equation $f\ast(\delta_a-\delta_0)=0$, where $\delta_x$ denotes the Dirac measure at $x$. When $\mathfrak{X}$ has the correct spectral properties, there is moreover a notion of generalized Fourier series for $f$. We will not use this property in the paper, but the preceding two properties of mean-periodic functions will be important.
\newline One pragmatic motivation for introducing mean-periodicity is that one can hope to obtain the most simple analytic properties of Hasse--Weil $L$-functions without first solving the far more difficult problem of automorphicity. Writing in $1997$, Langlands expressed some slightly unsatisfactory aspects of relying on the identification with automorphic $L$-functions to verify meromorphic continuation and functional equation \cite[Section~6]{WSFT}. His key point concerned special values of motivic (say, Hasse--Weil) $L$-functions outside of the half-plane in which their associated Dirichlet series converges: an automorphic representation has to have special properties in order to be motivic in nature, whereas the continuation of automorphic $L$-functions is completely uniform and does not consider such properties - if the only way we can formulate a continuation is in terms of the wider class of automorphic $L$-functions, how can we expect to make sense of problems that depend on the arithmetic geometry underlying only a special subclass? In any case, proving modularity results is hard work, and often requires restrictions on the algebraic variety and the base number field. Mean-periodicity, best formulated on the level of zeta functions, is a potential intermediate condition. Moreover, provided a powerful enough relationship, one could hope to use mean-periodicity as a stepping stone to the deep automorphicity results. This paper is intended as an initial search for such a relationship.
\section{Mean-Periodicity and Zeta Functions}
In this section we motivate and state the mean-periodicity correspondence of \cite{SRF}, which provides a bridge between zeta functions of arithmetic schemes and mean-periodic functions in appropriate functional spaces. 
\subsection{Mean-Periodic Strong Schwartz Functions}
As mentioned in the introduction, when dealing with mean-periodicity it is very important to fix the space in which a particular function is considered. It is known that the spaces $\mathscr{C}(\mathbb{R})$, $\mathscr{C}(\mathbb{R}^{\times}_{+})$ of continuous functions, or $\mathscr{C}^{\infty}(\mathbb{R})$, $\mathscr{C}^{\infty}(\mathbb{R}^{\times}_{+})$ of smooth functions, are not appropriate for the theory of zeta functions \cite[Remark~5.12]{SRF}. Instead, we will mostly work in the strong Schwartz space, following the inspiring work of Suzuki \cite{TDAAACAROGL2}. Moreover, this space appears elsewhere in the spectral theory of zeta functions \cite{ASIFTZOTRZF}, and, its dual, the space of (weak) tempered distributions has long been connected with the fundamental properties of Hecke $L$-Functions \cite{FZED}, \cite{OAROTICGRTPAZOLF}. 
\newline In order to avoid repetition of definitions, this section will frequently use that the exponential and logarithm maps provide topological isomorphisms between the locally compact abelian groups $\mathbb{R}$ and $\mathbb{R}^{\times}_{+}$.
\newline For all pairs of positive integers $(m,n)$ define a seminorm $|~|_{m,n}$ on the vector space of smooth functions $f:\mathbb{R}\rightarrow\mathbb{C}$ as follows
\[|f|_{m,n}=\sup_{x\in\mathbb{R}}|x^mf^{(n)}(x)|.\]
The Schwartz space $S(\mathbb{R})$ on the additive group $\mathbb{R}$ is the set of all smooth complex valued functions $f$ such that $|f|_{m,n}<\infty$, for all positive integers $m$ and $n$. The family of seminorms $\{|~|_{m,n}:m,n\in\mathbb{N}\}$ induce a topology on $S(\mathbb{R})$, with respect to which $S(\mathbb{R})$ is a Fr\'{e}chet space over $\mathbb{C}$.
\newline The Schwartz space on $\mathbb{R}_{+}^{\times}$, along with its topology, can be defined via the homeomorphism
\[S(\mathbb{R})\rightarrow S(\mathbb{R}_{+}^{\times})\]
\[f(t)\mapsto f(-\log(x)),\]
where $x=e^{-t}$.
\begin{definition}\label{strongschwartz.definition}
The strong Schwartz space on $\mathbb{R}_{+}^{\times}$ is the following set of functions:
\[\textbf{S}(\mathbb{R}_{+}^{\times}):=\bigcap_{\beta\in\mathbb{R}}\{f:\mathbb{R}_{+}^{\times}\rightarrow\mathbb{C}:[x\mapsto x^{-\beta}f(x)]\in S(\mathbb{R}^{\times}_{+})\},\]
with the topology given by the family of seminorms
\[||f||_{m,n}=\sup_{x\in\mathbb{R}_{+}^{\times}}|x^mf^{(n)}(x)|.\]
\end{definition}
Of course, one can then define $\textbf{S}(\mathbb{R})$ via the homeomorphism
\[f(x)\rightarrow f(e^{-t}),\]
where $t=\log(x)$. The multiplicative group $\mathbb{R}_{+}^{\times}$ acts on $\textbf{S}(\mathbb{R}_{+}^{\times})$ by
\[\forall x\in\mathbb{R}_{+}^{\times},f\mapsto\tau_x^{\times}f\]
\[\tau_x^{\times}f(y):=f(y/x).\]
Dual to the strong Schwartz space on $\mathbb{R}_{+}^{\times}$, we have the space $\textbf{S}(\mathbb{R}_{+}^{\times})^{\ast}$. When we endow the dual space with the weak $\ast$-topology, this is the space of weak-tempered distributions.
\newline Let $\phi$ be a weak-tempered distribution, and let $f$ be a strong Schwartz function. The natural pairing
\[<~,~>:\textbf{S}(\mathbb{R}_{+}^{\times})\times\textbf{S}(\mathbb{R}_{+}^{\times})^{\ast}\rightarrow\mathbb{C}\]
\[<f,\phi>=\phi(f),\]
gives rise to the convolution $f\ast\phi$, which is defined by
\[\forall x\in\mathbb{R}_{+}^{\times},(f\ast\phi)(x)=<\tau^{\times}_x\breve{f},\phi>,\]
where
\[\breve{f}(x)=f(x^{-1}).\]
Explicitly, the convolution is given by the following formula:
\[(f\ast\phi)(x)=\int_0^{\infty}f(x/y)\phi(y)\frac{dy}{y}.\]
We now give an abstract definition of mean-periodicity, which is valid in $\mathfrak{X}=\textbf{S}(\mathbb{R}_{+}^{\times})$
\begin{definition}\label{mp.definition}
Let $\mathfrak{X}$ be a locally convex, separated topological $\mathbb{C}$-vector space, with topological dual $\mathfrak{X}^{\ast}$. Assume that the Hahn-Banach theorem holds in $\mathfrak{X}$, and $\mathfrak{X}$ is equipped with an involution:
\[\check{~}:\mathfrak{X}\rightarrow\mathfrak{X}\]
\[f\mapsto\check{f}.\]
Let $G$ be a locally compact topological abelian group with a continuous representation:
\[\tau:G\rightarrow\text{End}(\mathfrak{X})\]
\[g\mapsto\tau_g.\]
For $f\in\mathfrak{X}$, consider the following topological space:
\[T(f)=\overline{\text{Span}_{\mathbb{C}}\{\tau_g(f):g\in G\}},\]
where the overline denotes topological closure. $f\in\mathfrak{X}$ is $\mathfrak{X}$-mean-periodic if either of the following two equivalent conditions holds:
\begin{enumerate}
\item $\mathcal{T}(f)\neq\mathfrak{X}$,
\item $\exists\psi\in\mathfrak{X}^{\ast}\backslash\{0\}$ such that $f\ast\psi=0$,
\end{enumerate}
where, for $\psi\in\mathfrak{X}^{\ast}$, the convolution $f\ast\psi:G\rightarrow\mathbb{C}$ is given by
\[(f\ast\psi)(g):=<\tau_g\check{f},\psi>,\]
and $<~,~>$ denotes the natural pairing between a topological space and its dual.
\end{definition}
In light of these definitions, we will introduce the following spaces
\[\mathcal{T}(f):=\overline{\text{Span}_{\mathbb{C}}(\{\tau_a^{\times}f:a\in\mathbb{R}_{+}^{\times}\})},\]
\[\mathcal{T}(f)^{\perp}:=\{\tau\in\textbf{S}(\mathbb{R}_{+}^{\times})^{\ast}:g\ast\tau=0,\forall g\in\mathcal{T}(f)\}.\]
\subsection{Functional Equations of Mellin Transforms}
Certain analytic properties of Mellin transforms can be inferred from the mean-periodicity of related functions. For example, consider a real-valued function
\[f:\mathbb{R}_{+}^{\times}\rightarrow\mathbb{R},\]
such that for all $A>0$,
\[f(x)=O(x^{-A}),\text{ as }x\rightarrow+\infty,\]
and for some $A>0$,
\[f(x)=O(x^A),\text{ as }x\rightarrow0^{+},\]
then its Mellin transform defines a holomorphic function in some half plane Re$(s)\gg0$:
\[M(f)(s):=\int_{0}^{+\infty}f(x)x^s\frac{dx}{x}.\]
One can ask when $M(f)$ admits a meromorphic continuation and a functional equation of the following form:
\[M(f)(s)=\varepsilon M(f)(1-s),\]
where $\varepsilon=\pm1$. Define
\[h_{f,\varepsilon}(x)=f(x)-\varepsilon x^{-1}f(x^{-1}),\]
then we see that
\[M(f)(s)=\int_1^{+\infty}f(x)x^s\frac{dx}{x}+\varepsilon\int_1^{+\infty}f(x)x^{1-s}\frac{dx}{x}+\int_0^1h_{f,\varepsilon}(x)x^s\frac{dx}{x}.\]
To avoid cumbersome notation, introduce
\[\varphi_{f,\varepsilon}(s):=\int_1^{+\infty}f(x)x^s\frac{dx}{x}+\varepsilon\int_1^{+\infty}f(x)x^{1-s}\frac{dx}{x},\]
\[\omega_{f,\varepsilon}(s):=\int_0^1h_{f,\varepsilon}(x)x^s\frac{dx}{x}.\]
Due to the first assumption on $f$, $\varphi_{f,\varepsilon}$ is an entire function satisfying
\[\varphi_{f,\varepsilon}(s)=\varepsilon\varphi_{f,\varepsilon}(1-s),\]
so, the meromorphic continuation and functional equation of $M(f)$ is equivalent to that of $\omega_{f,\varepsilon}$. We proceed by noting that $\omega_{f,\varepsilon}$ is in fact the Laplace transform of $h_{f,\varepsilon}(e^{-t})$:
\[\omega_{f,\varepsilon}(s)=\int_0^{\infty}h_{f,\varepsilon}(e^{-t})e^{-st}dt.\]
General theory of Laplace transforms relates the $\mathfrak{X}$-mean-periodicity of $h_{f,\varepsilon}$, for an appropriate\footnote{In this paper we will work with the Strong Schwartz space of functions $\mathbb{R}^{\times}_{+}\rightarrow\mathbb{C}$.} locally convex functional space $\mathfrak{X}$, to the meromorphic continuation of $\omega_{f,\varepsilon}$ to $\mathbb{C}$, given by the Mellin-Carleman transform of $h_{f,\varepsilon}$. Moreover, in this case we have the functional equation
\[MC(h_{f,\varepsilon})(s)=\varepsilon MC(h_{f,\varepsilon})(-s),\]
which is equivalent to
\[\omega_{f,\varepsilon}(s)=\varepsilon\omega_{f,\varepsilon}(1-s).\]
Up to rescaling the variable $s$ (see 2.6), the above outline will be applied to the inverse Mellin transform of
\[Z(\mathcal{S},\{k_i\},s)^m=\xi(\mathcal{S},s)^m\prod_{i=1}^n\xi(k_i,s/2)^m,\]
where $m,n$ are positive integers to be specified, $\mathcal{S}$ is a proper, regular model of a smooth projective curve $C$ over a number field $k$, and $\xi$ denotes its completed zeta function to be defined in the following subsection. For each $i$, $k_i$ is a finite extension of $k$. The meromorphic continuation and functional equation (up to sign) of $Z(\mathcal{S},\{k_i\},s)^{m}$ is equivalent to that of $\zeta(\mathcal{S},s)$. 
\subsection{Zeta Functions and $L$-Functions}
Let $\mathcal{S}$ be a scheme of finite type over $\mathbb{Z}$. For example, an algebraic variety over a finite field, or a proper, regular model of an algebraic variety over a number field. The zeta function of $\mathcal{S}$ is defined, for $\Re(s)>\dim(\mathcal{S})$ by
\[\zeta(\mathcal{S},s)=\prod_{x\in\mathcal{S}_0}\frac{1}{1-|k(x)|^{-s}},\]
where $\mathcal{S}_0$ denotes the set of closed points of $\mathcal{S}$. For a based scheme $\mathcal{S}\rightarrow\Spec(\mathcal{O}_k)$, we have the ``fibral'' decomposition
\[\zeta(\mathcal{S},s)=\prod_{\mathfrak{p}\in\text{Spec}(\mathcal{O}_k)}\zeta(\mathcal{S}_{\mathfrak{p}},s),\]
each $\mathcal{S}_{\mathfrak{p}}$ being a, possibly non-smooth, algebraic variety over the finite field $\mathcal{O}_k/\mathfrak{p}$. The zeta function of such a variety is a rational function in the variable $p^{-s}$, where $p$ is the characteristic of this finite field. 
\newline For example, if $C$ is a smooth, projective curve over a number field $k$ and $\mathcal{S}$ is a proper regular model over $\mathcal{O}_k$ and if $\mathfrak{p}$ is a good prime, then we have
\[\zeta(\mathcal{S}_{\mathfrak{p}},s)=\frac{L_{\mathfrak{p}}(H^0(\mathcal{S}_{\mathfrak{p}}),s)L_{\mathfrak{p}}(H^2(\mathcal{S}_{\mathfrak{p}}),s)}{L_{\mathfrak{p}}(H^1(\mathcal{S}_{\mathfrak{p}}),s)},\]
where, for $i=0,1,2$, each $L$-function is given by a characteristic polynomial:
\[L_{\mathfrak{p}}(H^i(\mathcal{S}_{\mathfrak{p}}),s)=\det(1-f_{\mathfrak{p}}N\mathfrak{p}^{-s}|H^i_{\text{\'{e}t}}(\mathcal{S}_{\mathfrak{p}},\mathbb{Q}_l)),\]
where $l\neq p$ is a prime upon which the polynomial does not depend. At a bad prime, one takes inertia invariants of \'{e}tale cohomology, as explained in \cite{FLDFZDVA}. 
\newline One defines, for $\Re(s)>i+1$,
\[L(H^i(\mathcal{S}),s)=\prod_{\mathfrak{p}}L_{\mathfrak{p}}(H^i(\mathcal{S}_{\mathfrak{p}}),s),\]
and for $\Re(s)>2$, one has the identity
\[\zeta(\mathcal{S},s)=\frac{L(H^0(\mathcal{S}),s)L(H^2(\mathcal{S}),s)}{L(H^1(\mathcal{S}),s)}\]
\newline Let $L(C,s):=L(J(C),s)$, where $J(C)$ is the Jacobian of $C$, which is an abelian variety. One uses \'{e}tale cohomology to see that the function
\[n(\mathcal{S},s)=\zeta(\mathcal{S},s)\cdot\bigg(\frac{\zeta(k,s)\zeta(k,s-1)}{L(C,s)}\bigg)^{-1}\]  
is a finite product of functions rational in variables of the form $p^{-s}$, depending on (the bad fibres of) our choice of model $\mathcal{S}$ from its birational class. The meromorphic continuation of $\zeta(\mathcal{S},s)$ is thus equivalent to that of $\frac{\zeta(k,s)\zeta(k,s-1)}{L(C,s)}$, and hence $L(C,s)$. We will denote the quotient by $\zeta(C,s)$, which does not depend on a choice of model.
\[\zeta(C,s):=\frac{\zeta(k,s)\zeta(k,s-1)}{L(C,s)}\]
As for the functional equation, we must define ``completed'' zeta functions $\xi(\mathcal{S},s)$. We will do so by considering the quotient $\zeta(C,s)$. The functional equation of $L(C,s)$ is of the form
\[\Lambda(C,s)=\pm\Lambda(C,2-s),\]
where $\Lambda(C,s)=L(C,s)A(C)^{s/2}\Gamma(C,s)$ is defined in \cite{FLDFZDVA}. $A(C)$ denotes the conductor of $C$, which accounts for ramification, and $\Gamma(C,s)$ is the gamma factor, which accounts for the Hodge structure of the Betti cohomology at archimedean places of the base number field. The numerator, a product of Dedekind zeta functions $\zeta(k,s)\zeta(k,s-1)$, satisfies a functional equation with respect to $s\mapsto2-s$, involving the completion $\xi(k,s)=\Gamma(k,s)\zeta(k,s)$. The Gamma factor is the usual one for Dedekind zeta functions, which also has an interpretation via the Hodge structure of Betti cohomology.
\newline  According to \cite[Lemma~1.2]{DRCACOC}, $n(\mathcal{S},s)$ admits a functional equation with respect to $s\mapsto2-s$, so that the functional equation of $\zeta(\mathcal{S},s)$ is equivalent to that of $\zeta(C,s)$, which is that the following quotient is invariant, up to sign, under $s\mapsto2-s$:
\[\xi(C,s):=\frac{\xi(k,s)\xi(k,s-1)}{\Lambda(k,s)}.\]
\subsection*{Notation}
The $L$-functions of curves and the zeta functions of their models are intimately connected and in this paper we will pass between the two quite freely. We will use the notation $\zeta(\mathcal{S},s)$ (resp. $\xi(\mathcal{S},s)$) for the zeta function of an arithmetic surface $\mathcal{S}$ (resp. its completion). The notation $L(C,s)$ (resp. $\Lambda(C,s)$) will be used for the Hasse--Weil $L$-function of its generic fibre (resp. its completion). We also have the Hasse--Weil zeta function $\zeta(C,s)$ (and its completion $\xi(C,s)$) of $C$, which depends only on the curve $C$, not any particular choice of model. We will always specify when something depends on a choice of model, for example, $\Gamma(\mathcal{S},s)$ (resp. $A(\mathcal{S})$) denoted the gamma factor (resp. conductor) of $\zeta(\mathcal{S},s)$, the analogues for $L(C,s)$ will be denoted by $A(C)$ and $\Gamma(C,s)$.
\subsection{Abstract Mean-Periodicity Correspondence}
We begin by abstracting the expected analytic properties of the zeta functions of arithmetic schemes. This is broadly comparable to the notion of the Selberg class of $L$-functions \cite{OANCARAACODS}. Quotients of $L$-functions are not generally in the Selberg class and are not expected to be automorphic. 
\begin{definition}
Let $\mathfrak{Z}$ be a complex valued function defined in some half plane $\Re(s)>\sigma_1$, which has a decomposition
\[\mathfrak{Z}(s)=\frac{\mathscr{L}_1(s)}{\mathscr{L}_2(s)}.\]
We will say that $\mathfrak{Z}$ is of ``expected analytic shape'' if, for $i=1,2$, the following hold:
\begin{enumerate}
\item $\mathscr{L}_i(s)$ is an absolutely convergent Dirichlet series in $\Re(s)>\sigma_1$ and has meromorphic continuation to $\mathbb{C}$.
\item There exist $r_i\geq1$ such that, for $1\leq j\leq r_i$, there are $\lambda_{i,j}>0$ with $\mu_{i,j}\in\mathbb{C}$ such that $\Re(\mu_{i,j})>\sigma_1\lambda_{i,j}$ and there are $q_i>0$ such that the function
\[\widehat{\mathscr{L}}_i(s):=\gamma_i(s)\mathscr{L}_i(s),\]
where
\[\gamma_i(s)=q_i^{s/2}\prod_{j=1}^{r_1}\Gamma(\lambda_{i,j}s+\mu_{i,j}),\]
satisfies
\[\widehat{\mathscr{L}}_i(s)=\epsilon_i\overline{\widehat{\mathscr{L}}_i(1-s)},\]
for some $\varepsilon_i\in\mathbb{C}$ such that $|\varepsilon_i|=1$.
\item There exists a polynomial $P(s)$ such that $P(s)\widehat{\mathscr{L}}_i(s)$ is an entire function of order one.
\item The logarithmic derivative of $\mathscr{L}_2(s)$ is an absolutely convergent Dirichlet series in some right half plane $\Re(s)>\sigma_2\geq\sigma_1$.
\end{enumerate}
When $\mathfrak{Z}$ is of expected analytic shape, one defines its completion as
\[\widehat{\mathfrak{Z}}(s)=\frac{\widehat{\mathscr{L}}_1(s)}{\widehat{\mathscr{L}}_2(s)}.\]
We will further assume that all poles of $\widehat{\mathfrak{Z}}$ lie within some vertical strip $|\Re(s)-\frac{1}{2}|\leq\omega$.
\end{definition}
For example, in the notation of the previous subsection, the zeta function of an arithmetic surface $\zeta(\mathcal{S},s)$ is of expected analytic shape if each of its Hasse--Weil $L$-factors $L(H^i(C),s)$ satisfy the conjectures in \cite{FLDFZDVA}, and in this case $\varepsilon=\pm1$, for details see \cite[Remark~5.20]{SRF}.
\newline For $i=1,\dots,m$, let $k_i$ denote a finite extension of $k$. We will be interested in the inverse Mellin transform of functions of the form
\[\widehat{\mathfrak{Z}}(ds)\prod_{i=1}^m\xi(k_i,s),\]
where $d\in\mathbb{N}$ is to be specified and $\xi(k_i,s)$ is the completed Dedekind function of $k_i$. Explicitly let $c>\frac{1}{2}+\omega$, the product above is then analytic to the right of $c$ and the inverse Mellin transform is the contour integral
\[f(\mathfrak{Z},\{k_i\},x):=\frac{1}{2\pi i}\int_{(c)}\widehat{\mathfrak{Z}}(s)\prod_{i=1}^m\xi(k_i,s)x^{-s}ds,\]
where $(c)$ denotes the vertical line $\Re(s)=c$.
\begin{definition}\label{boundaryfunction.definition}
Let $\varepsilon=\frac{\varepsilon_1}{\varepsilon_2}$. The boundary function\footnote{The terminology ``boundary function'' pertains to the idea that these functions should be integrals over topological boundaries of subsets of the related adelic space - this idea will not be used in this paper.} associated to $(\mathfrak{Z},\{k_i\})$ is
\[h(\mathfrak{Z},\{k_i\}):\mathbb{R}_{+}^{\times}\rightarrow\mathbb{C}\]
\[h(\mathfrak{Z},\{k_i\},x):=f(\mathfrak{Z},\{k_i\},x)-\varepsilon x^{-1}f(\mathfrak{Z},\{k_i\},x^{-1}),\]
In additive language, we have
\[H(\mathfrak{Z},\{k_i\}):\mathbb{R}\rightarrow\mathbb{C}\]
\[H(\mathfrak{Z},\{k_i\},t)=h(\mathfrak{Z},\{k_i\},e^{-t}).\]
\end{definition}
For a clean statement of the mean-periodicity correspondence, we need the following technical condition.
\begin{definition}
For $i=1,\dots,m$, let $k_i$ be a finite extension of $k$, and denote by $\xi(k_i,s)$ the completed Dedekind zeta function of $k_i$. Let $\gamma(s)$ be a meromorphic function on $\mathbb{C}$. We will say that $(\gamma(s),\{k_i\})$ is ``nice'' if there is some $t\in\mathbb{R}$ such that for all $a\leq b\in\mathbb{R}$, all $\Re(s)\in[a,b]$, and all $|\Im(s)|\geq t$
\[\gamma(s)\prod_{i=1}^n\xi(k_i,s)\ll_{a,b,t}|\Im(s)|^{-1-\delta}.\]
for some $\delta>0$.
\end{definition}
The mean-periodicity correspondence is encapsulated in the following theorem:
\begin{theorem}\label{MPgeneral.theorem}
Let $\mathfrak{Z}(s)$ be a complex valued function defined in some half plane $\Re(s)>\sigma_1$, which has a decomposition
\[\mathfrak{Z}(s)=\frac{\mathscr{L}_1(s)}{\mathscr{L}_2(s)},\]
then
\begin{enumerate}
\item If $\mathfrak{Z}(s)$ is of expected analytic shape then there exists an $n_{\mathfrak{Z}}\in\mathbb{Z}$ such that, if $\{k_i\}_{1\leq i\leq n}$ is a finite set of (not necessarily distinct) number fields for any $n\geq n_{\mathfrak{Z}}$, then $h(\mathfrak{Z},\{k_i\},x)$ is $\mathscr{C}_{\text{poly}}^{\infty}(\mathbb{R}_{+}^{\times})$- and $\textbf{S}(\mathbb{R}_{+}^{\times})$-mean-periodic. Similarly for $H(\mathfrak{Z},\{k_i\},t)$.
\item If, for some $n\in\mathbb{Z}$, we have number fields $k_i,i=1,...,n$ such that $(\frac{\gamma_1(s)}{\gamma_2(s)},\{k_i\})$ is nice, and $h(\mathfrak{Z},\{k_i\},x)$ is $\mathscr{C}_{\text{poly}}^{\infty}(\mathbb{R}_{+}^{\times})-$ or $\textbf{S}(\mathbb{R}_{+}^{\times})$-mean-periodic, then $\mathfrak{Z}(s)$ extends to a meromorphic function on $\mathbb{C}$ and
\[\mathfrak{Z}(s)=\pm\mathfrak{Z}(1-s).\]
Similarly for $H(\mathfrak{Z},\{k_i\},t)$.
\end{enumerate}
\end{theorem}
\begin{proof}
This is an immediate consequence of \cite[Theorem~5.18]{SRF}.
\end{proof}
We will not spend any time trying to optimise this statement in terms of $k_i$ and $n_{\mathfrak{Z}}$. By rescaling the argument of the zeta functions of arithmetic schemes, we can apply the above result to the problem of their meromorphic continuation and functional equation. We will spell this out in the following subsection. 
\newline Let $\mathscr{S}(\mathcal{O}_k)$ denote the set of arithmetic schemes over the ring of integers $\mathcal{O}_k$ in a number field $k$ whose zeta functions have the expected analytic properties, and the $\text{MP}(\mathfrak{X})$ denote the mean-periodic functions on $\mathbb{R}^{\times}_{+}$ in an appropriate function space. We have a family of maps
\[\mathscr{H}(\{k_i\}_{i=1}^{n\geq n_{\mathcal{S}}}):\mathscr{S}(\mathcal{O}_k)\rightarrow\text{MP}(\mathfrak{X})\]
\[\mathcal{S}\mapsto h(\mathcal{S},\{k_i\}_{i=1}^{n}):\mathbb{R}^{\times}_{+}\rightarrow\mathbb{C}.\]
An open problem is to find the image of the maps $\mathscr{H}$. 
\newline There is a completely analogous theorem over function fields, which we will not pursue here. In this case it is known that mean-periodicity would follow from rationality of the boundary term. 
\subsection{Mean-Periodicity for Arithmetic Surfaces}
In practice, we will only work with proper, regular models $\mathcal{S}$ of smooth, projective geometrically connected algebraic curves $C$ over a number fields $k$. As discussed in 2.4, we have
\[\xi(\mathcal{S},s)=\pm\xi(\mathcal{S},2-s)\Leftrightarrow\xi(C,s)=\pm\xi(C,2-s).\]
To remove the ambiguity of the sign, we will consider the squares\footnote{Another reason for doing this is that it is the square of the zeta function which admits an interpretation in terms of two-dimensional adeles, which was put forward as a way of proving the mean-periodicity condition in \cite{F3}, \cite{MeZetaintegrals}.}:
\[\xi(\mathcal{S},s)^2=\xi(\mathcal{S},2-s)^2\Leftrightarrow\xi(C,s)^2=\xi(C,2-s)^2.\]
$\xi(C,s)^2$ is the completion of the squared Hasse--Weil zeta function of $C$, which is a quotient of Dirichlet series
\[\zeta(C,s)^2=\frac{\zeta(k,s)^2\zeta(k,s-1)^2}{L(C,s)^2},\]
however the expected functional equations of the numerator and denominator are not of the correct form to apply the mean-periodicity correspondence. Indeed, put
\[\mathscr{L}_1(C,s)=\zeta(k,s)^2\zeta(k,s-1)^2,\]
\[\mathscr{L}_2(C,s)=L(C,s)^2,\]
then we have
\[\gamma_1(s)=\Gamma(k,s)^2\Gamma(k,s-1)^2\]
\[\gamma_2(s)=A(C)^s\Gamma(E,s)\]
such that, if $\widehat{\mathscr{L}_i}(C,s)=\gamma_i(s)\mathscr{C}(s)$, for $i=1,2$, we expect\footnote{In the first case, this is not just an expectation, but a theorem.}
\[\widehat{\mathscr{L}_1}(C,s)=\widehat{\mathscr{L}_1}(C,2-s),\]
\[\widehat{\mathscr{L}_2}(C,s)=\widehat{\mathscr{L}_2}(C,2-s).\]
To resolve this, we need to rescale $s$ so that the functional equation are with respect to $s\mapsto1-s$. To that end, we will in fact consider the following quotient of Dirichlet series
\[\zeta(C,s/2+1/4)^2.\]
The next definition will be used in sections 3, 4 and 5.
\begin{definition}
Let $C$ be a smooth projective curve over a number field $k$, then define
\[h(C,\{k_i\},x):\mathbb{R}^{\times}_{+}\rightarrow\mathbb{C}\]
\[h(C,\{k_i\},x)=f(C,\{k_i\},x)-x^{-1}f(C,\{k_i\},x^{-1})\]
where
\[f(C,\{k_i\},x)=\frac{1}{2\pi i}\int_{(c)}\xi(C,s+1/2)^2(\prod_{i=1}^m\xi(k_i,s/2+1/4)^2)x^{-s}ds.\]
\end{definition}
Theorem~\ref{MPgeneral.theorem} reduces to the following.
\begin{corollary}\label{MPMCFE.corollary}
Let $C$ be a smooth, projective curve over a number field $k$.
\begin{enumerate}
\item Assume that $\zeta(C,s)$ admits meromorphic continuation to $\mathbb{C}$ and satisfies the functional equation $\xi(C,s)^2=\xi(C,2-s)^2$, the logarithmic derivative of $L(C,s)$ is an absolutely convergent Dirichlet series in the right half plane $\Re(s)>1$ and there exists a polynomial $P(s)$ such that $P(s)L(C,s)$ is an entire function on $\mathbb{C}$ of order $1$. Then, there exists $m_{C}\in\mathbb{N}$ such that for all sets $\{k_i\}$ of $m\geq m_C$ number field extensions of $k$, $h(C,\{k_i\},x)$ is $\textbf{S}(\mathbb{R}^{\times}_{+})$-mean-periodic.
\item Conversely, suppose that there exists $m_{\mathcal{S}}\in\mathbb{N}$ such that, for some set $\{k_i\}$ of $m_{C}$ finite extensions of $k$, the function $h(C,\{k_i\},x)$ is $\textbf{S}(\mathbb{R}^{\times}_{+})$-mean-periodic and that, for some $\delta>0$ we have $\Gamma(C,2s)\prod_{i=1}^{m_{C}}\xi(k_i,s)\ll|t|^{-1-\delta},$ then $\zeta(C,s)$ (hence $L(C,s)$) admits meromorphic continuation and satisfies the functional equation $\xi(C,s)^2=\xi(C,2-s)^2.$
\end{enumerate}
\end{corollary}
We are lead to make the following conjecture:
\begin{conjecture}\label{mphypothesis.conjecture}
Let $C$ be a smooth, projective curve over a number field $k$. There exists $n_C$ such that, for any finite set $\{k_i\}$ of $n\geq n_C$ finite number field extensions of $k$, $h(C,\{k_i\},x)$ is $\textbf{S}(\mathbb{R}^{\times}_{+})$-mean-periodic.
\end{conjecture}
\subsection{Expansion of Boundary Function}
In the following sections, we will specifically work with that case $n=1$, $k_1=k$. We will label this definition for future reference.
\begin{definition}\label{hC.definition}
Let $C$ be a smooth projective curve over $k$, define
\[h_C(x):\mathbb{R}^{\times}_{+}\rightarrow\mathbb{C}\]
\[h_C(x):=h(C,\{k\},x).\]
\end{definition}
In this paper we will show that the $\textbf{S}(\mathbb{R}^{\times}_{+})$-mean-periodicity of $h_C(x)$ can be deduced from the expected automorphicity of $C$. $h_C$ is defined in terms of the inverse Mellin transform of a product of zeta functions. This leads to the following expansion of $h_C(x)$.
\begin{lemma}\label{expansion.lemma}
\[h_{C}(x)=\lim_{T\rightarrow\infty}\sum_{\text{Im}(\lambda)\leq T}\sum_{m=0}^{m_{\lambda}-1}C_{m+1}(\lambda)\frac{(-1)^{m}}{m!}x^{-\lambda}(\log(x)^m),\]
where $\lambda$ runs over the poles of $\xi(k,\frac{s}{2}+\frac{1}{4})^2\xi(C,s+\frac{1}{2})^2$, $m_{\lambda}$ denotes the multiplicity of $\lambda$ and $C_{m}(\lambda)$ is the coefficient of $(s-\lambda)^{-m}$ in the principal part of $\xi(k,\frac{s}{2}+\frac{1}{4})^2\xi(C,s+\frac{1}{2})^2$ at $s=\lambda$.
\end{lemma}
This result can be found in \cite[Sections~4,~5]{SRF}. We sketch the proof for intuition.
\begin{proof}
The inverse Mellin transform is a contour integral. We may evaluate this by taking the limit of integrals over rectangles in the complex plane, it is these integrals which provide the expressions in the sum.
\end{proof}
\begin{remark}
Properties of the function $h_C$ are also related to other important and far from understood properties of $L(C,s)$. For example, the GRH is studied in \cite{POCFAWAOES}, following the outline in \cite{F3}.
\end{remark}
\section{Hecke Characters}\label{Hecke.section}
In the following section, we will construct convolutors for $h_C$ from the automorphic representations associated with algebraic curves. The technique will not differ greatly from that of this section, which we include so as to demonstrate the method in the most simple context, that of CM elliptic curves. Unlike what follows, the results of this section are unconditional on account of the fact that the analytic and automorphic properties were deduced decades ago (see \cite[Chapter~2]{S}). Underlying the results described here is a result about Hecke characters, which could be useful for future examination of the twisted mean-periodicity correspondence and converse theorems for meromorphic continuation. 
\newline Let $k$ be a number field, and recall that $k^{\times}$ diagonally embeds into $\mathbb{A}_k^{\times}$. We will use the adelic formulation of Hecke characters.
\begin{definition}
A Hecke character of a number field $k$ is a continuous homomorphism $\mathbb{A}_k^{\times}\rightarrow\mathbb{C}$ which is trivial on $k^{\times}$.
\end{definition}
Let $E$ be an elliptic curve over $k$. There is always an inclusion of the ring $\mathbb{Z}$ of integers into the endomorphism ring $\text{End}(E)$ defined by \[n\mapsto[n]:P\mapsto nP=\overbrace{P+\dots+P}^{n\text{ times}}.\]
In most cases, this inclusion is actually an equality. In general, it is easy to deduce that $\text{End}(E)$ is isomorphic to $\mathbb{Z}$ or an order $\mathcal{O}$ in a quadratic imaginary field $F$ \cite[Chapter~2]{S}. When $\text{End}(E)=\mathcal{O}$, we say that $E$ has complex multiplication (``CM'') by $\mathcal{O}$ and we have an isomorphism of $\mathbb{Q}$-algebras $F\cong\text{End}(E)\otimes_{\mathbb{Z}}\mathbb{Q}$, motivating the terminology ``CM by $F$''.
\newline Let $E$ be an elliptic curve over a number field $k$ with CM by $F$, which we assume, for now, to be a subfield of $k$.\footnote{The case in which $K$ is not a subfield of $k$ is completely similar.} By the first main theorem of CM, \cite[Chapter~2,~Theorem~8.2]{S}, for all $x\in\mathbb{A}_k^{\times}$ there is a unique $\alpha_x\in F^{\times}$ such that $\alpha_x\mathcal{O}_F$ is the principal fractional ideal generated by $s=N_F^kx\in\mathbb{A}_F^{\times}$ and, for all fractional ideals $\mathfrak{a}\subset F$ and analytic isomorphisms
\[f:\mathbb{C}/\mathfrak{a}\rightarrow E(\mathbb{C}),\]
the following diagram commutes
\[\xymatrix{F/\mathfrak{a}\ar@{->}^{\alpha_xs^{-1}}[r]\ar@{->}_f[d] & F/\mathfrak{a}\ar@{->}^{f~~~~,}[d] \\
E(k^{\text{ab}})\ar@{->}^{r_k(x)}[r] & E(k^{\text{ab}})
 }\]
where $r_k:\mathbb{A}_k^{\times}\rightarrow\text{Gal}(k^{\text{ab}}/k)$ is the reciprocity map. Thus we have a homomorphism
\[\alpha_{E/k}:\mathbb{A}_k^{\times}\rightarrow F^{\times}\hookrightarrow\mathbb{C}^{\times}\]
\[\alpha_{E/k}(x)=\alpha_x\]
The Hecke character associated to the CM elliptic curve $E/k$ is then defined as
\[\psi_{E/k}:\mathbb{A}_k^{\times}\rightarrow\mathbb{C}^{\times}\]
\[\psi_{E/k}(x)=\alpha_{E/k}(x)N^k_F(x^{-1})_{\infty},\]
where, $N^k_F:\mathbb{A}_k^{\times}\rightarrow\mathbb{A}_F^{\times}$ denotes the idelic norm map
\[(\alpha_v)\mapsto((\prod_{w|v}N^{k_w}_{F_v}\alpha_w)_v),\] 
and, for any idele $\beta\in\mathbb{A}_F^{\times}$, $\beta_{\infty}\in\mathbb{C}^{\times}$ is the component of $\beta$ corresponding to the unique absolute value on $K$ (which is a quadratic imaginary field).
\newline The $L$-function of an elliptic curve $E$ over a number field $k$ with CM by the quadratic imaginary field $F$ is given by
\[L(E,s)=\begin{cases}
L(s,\psi_{E/k})L(s,\overline{\psi_{E/k}}), & \text{if }F\subseteq k,\\
L(s,\psi_{E/k'}), & \text{otherwise,}
\end{cases}\]
where $k'=kF$ in the second case.
\newline On the locally compact group $\mathbb{A}_k$ we have a Haar measure inducing surjective ``module'' map on the idele group $\mathbb{A}_k^{\times}$:
\[|~|:\mathbb{A}_k^{\times}\rightarrow\mathbb{R}_{+}^{\times},\]
which may be chosen so that $|x|=1$, for all $x\in k^{\times}$.
\newline For a positive real number $x\in\mathbb{R}_{+}^{\times}$, we introduce the following notation:
\[(\mathbb{A}_k^{\times})_x:=\{\alpha\in\mathbb{A}_k^{\times}:|\alpha|=x\}.\]
Let $S(\mathbb{A}_k)$ be the adelic Schwartz-Bruhat space, that is, factorizable functions $f:\mathbb{A}_k\rightarrow\mathbb{C}$ whose archimedean components are Schwartz and whose non-archimedean components are locally constant, compactly supported and almost always the characteristic function of the maximal compact subring.
\newline For a Hecke character $\chi:\mathbb{A}_k^{\times}\rightarrow\mathbb{C}^{\times}$, we can define a map:
\[Z(\chi,\cdot):S(\mathbb{A}_k^{\times})\rightarrow\textbf{S}(\mathbb{R}^{\times}_{+})\]
\[f\mapsto Z(\chi,f),\]
where
\[Z(\chi,f)(x)=\int_{(\mathbb{A}_k^{\times})_x}f(\alpha)\chi(\alpha)d\alpha.\]
This integral converges absolutely and its image is contained in $\textbf{S}(\mathbb{R}^{\times}_{+})$ because for any integer $N$ there exists a positive constant $C$ such that, for all $x\in\mathbb{R}^{\times}_{+}$,
\[|Z(\chi,f)(x)|\leq Cx^{-N}.\]
We will denote the image of $Z(\chi,f)$ by $\mathcal{V}_{\chi}\subseteq\textbf{S}(\mathbb{R}^{\times}_{+})$.
\newline There is a well known relationship between $Z(\chi,f)(x)$ and the Hecke $L$-function $L(s,\chi)$, namely, there exists a non-zero holomorphic function $\Psi_{f,\chi}(s)$ depending on $f$ and $\chi$ such that
\[\int_0^{\infty}Z(\chi,f)(x)x^s\frac{dx}{x}=\int_{\mathbb{A}_k^{\times}}f(\alpha)\chi(\alpha)|\alpha|^sd^{\times}\alpha=\Psi_{f,\chi}(s)L(s,\chi),\]
where $d^{\times}\alpha$ is the measure on the idele group $\mathbb{A}_k^{\times}$. In definition~\ref{hC.definition} we introduced the strong Schwartz function $h_C(x):\mathbb{R}^{\times}_{+}\rightarrow\mathbb{C}$, which we will now study in the case $C=E$, an elliptic curve over $k$, with CM by $F$. Let $\mathcal{T}(h_{E})$ and $\mathcal{T}(h_{E})^{\perp}$ be as in $2.1$ and recall that, if the set of non-zero convolutors $\mathcal{T}(h_{E})^{\perp}\backslash\{0\}$ is non-empty, then $h_{E}$ is mean-periodic. In this section, we will construct a non-trivial subset of $\mathcal{T}(h_{E})^{\perp}$ from $\mathcal{V}_{\chi}$, where $\chi=\psi_{E/k}$ is the associated Hecke character. The fact that the boundary term is mean-periodic could already have been established from the analytic properties it implies, however the explicit connection to Hecke characters is new. First, define
\[w_0(x)=\frac{1}{2\pi i}\int_{(c)}A(E)^{s}\Gamma(E,s)^2s^4(s-2)^4(s-1)^2x^{-s}ds,\]
where $A(E)^{s}\Gamma(E,s)^2$ is the product of the gamma factor and conductor appearing in the functional equation of $\Lambda(E,s)$. If $\psi_{E/k}$ denotes the Hecke character associated to the CM elliptic curve $E$ over $k$ and $k'=kK$, then define
\[\mathcal{W}_{\psi_{E/k}}=\begin{cases}
w_0\ast\mathcal{V}_{\psi_{E/k}}\ast\mathcal{V}_{\psi_{E/k}}\ast\mathcal{V}_{\overline{\psi}_{E/k}}\ast\mathcal{V}_{\overline{\psi}_{E/k}}, & K\subset k \\
w_0\ast\mathcal{V}_{\psi_{E/k'}}\ast\mathcal{V}_{\psi_{E/k'}}, & \text{otherwise.}
 \end{cases}\]
where we use the following shorthand:
\[\mathcal{U}\ast\mathcal{V}:=\overline{\text{Span}_{\mathbb{C}}\{u\ast v:u\in\mathcal{U},v\in\mathcal{V}\}}.\]
\begin{theorem}\label{MPCM.theorem}
Let $E/k$ be an elliptic curve with CM by $K$, then 
\[\mathcal{W}_{\psi_{E/k}}\subset\mathcal{T}(h_E)^{\perp},\]
where $\psi_{E/k}$ is the Hecke character associated to $E$ and $\mathcal{W}_{\psi_{E/k}}$ is as above.
\end{theorem}
\begin{proof}
First assume that $F\subset k$, we will put $\psi:=\psi_{E/k}$. Let $w$ be any function in $\mathcal{W}_{\psi}$ - we will show that its convolution with the boundary function is $0$.
\newline By lemma~\ref{expansion.lemma}, we have
\[h_E(x)=\lim_{T\rightarrow\infty}\sum_{\text{Im}(\lambda)\leq T}\sum_{m=0}^{m_{\lambda}-1}C_{m+1}(\lambda)\frac{(-1)^{m}}{m!}f_{\lambda,m}(x),\]
where 
\[f_{\lambda,m}(x)=x^{-\lambda}(\log(x)^m),\]
and the poles $\lambda$ of $\xi(k,\frac{s}{2}+\frac{1}{4})^2\xi(E,s+\frac{1}{2})^2$ must be one of the following three cases:
\begin{enumerate}
\item A pole of $\xi(k,s/2)$,
\item A pole of $\xi(k,s)\xi(k,s-1)$, the numerator of $\xi(E,s)$,
\item A zero of $\Lambda(E,s)$, the denominator of $\xi(E,s)$.
\end{enumerate}
In the first two cases, we can place the poles as occurring at $0,1,2$. Using this expansion, we have, for $W\in\mathcal{W}_{\psi}$,
\[W\ast h_{E}(x)=\lim_{T\rightarrow\infty}\sum_{\text{Im}(\lambda)\leq T}\sum_{m=0}^{m_{\lambda}-1}C_{m+1}(\lambda)\frac{(-1)^{m}}{m!}W\ast f_{\lambda,m}(x).\]
We evaluate each convolution in this series. Firstly,
\[W\ast f_{\lambda,m}(x)=\int_0^{\infty}W(y)f_{\lambda,m}(x/y)\frac{dx}{x}.\]
It is elementary to deduce that
\[W\ast f_{\lambda,m}(x)=\sum_{i=1}^m(-1)^j\binom{m}{j}x^{-\lambda}(\log(x))^{m-j}\int_0^{\infty}W(y)y^{\lambda}(\log(y))^j\frac{dy}{y}.\]
Next we observe that
\[\int_0^{\infty}W(y)y^{\lambda}(\log(y))^j\frac{dy}{y}=\frac{d^j}{d\lambda^j}\int_0^{\infty}W(y)y^{\lambda}\frac{dy}{y}.\]
$W$ is a limit of a linear combination of functions of the form 
\[w=w_0\ast Z(\psi,f_1)\ast Z(\psi,f_2)\ast Z(\overline{\psi},f_3)\ast Z(\overline{\psi},f_4),\]
where, for $i=1,\dots,4$, $f_i\in S(\mathbb{A}_F)$. For such a function
\[\int_{0}^{\infty}w(y)y^{\lambda}\frac{dy}{y}=\int_0^{\infty}w_0(y)y^{\lambda}\frac{dy}{y}\prod_{i=1}^2\int_{0}^{\infty}Z(\psi,f_i)(y^{\lambda})\frac{dy}{y}\prod_{i=3}^4\int_0^{\infty}Z(\overline{\psi},f_i)(y^{\lambda})\frac{dy}{y}.\]
Recall that, for each $i=1,\dots,4$ there exists a holomorphic function $\Psi_{f_i,\psi_i}$ such that, for a Hecke character $\chi$,
\[\int_0^{\infty}Z(\chi,f_i)(y^{\lambda})\frac{dy}{y}=\Psi_{f_i,\psi_i}(\lambda)L(\lambda,\psi_i).\]
Finally, observe that, by Mellin inversion
\[\int_0^{\infty}w_0(y)y^{\lambda}\frac{dy}{y}=\Gamma(E,\frac{\lambda}{2})^2\lambda^4(\lambda-2)^4(\lambda-1)^2.\]
Altogether
\[\int_{0}^{\infty}w(y)y^{\lambda}\frac{dy}{y}=A(E)^{\lambda}\Gamma(E,\lambda)^2\lambda^4(\lambda-2)^4(\lambda-1)^2L(E,\lambda)^2\prod_{i=1}^4\Psi_{f_i,\psi_i}(s)\]
\[=\Lambda(E,\lambda)^2\lambda^4(\lambda-2)^4(\lambda-1)^2\prod_{i=1}^4\Psi_{f_i,\psi_i}(\lambda).\]
By the classification of poles above, upon rescaling $\lambda$, this expression is $0$ whenever $\lambda$ is a pole. The case $F\not\subset k$ is completely similar.
\end{proof}
In the more general framework of the next section we will develop some inclusions in the opposite direction. We will also study an equivalent adelic construction.
\section{$\GL_n$}\label{GL.section}
Let $C$ be a smooth projective curve of genus $g$ over a number field $k$. The $L$-function of $C$ is expected to arise as the $L$-function of a GL$_{2g}(\mathbb{A}_k)$ automorphic representation. Our aim is to investigate the relationship between mean-periodicity and automorphicity.
\newline \textit{In this section, we will assume that there exists an algebraic cuspidal automorphic representation }$\pi$\textit{ of }$\GL_{2g}(\mathbb{A}_k)$\textit{ such that}
\[L(C,s)=L(\pi,s-\frac{1}{2}).\]
We will denote GL$_{2g}$ by $G$. If
\[G(\mathcal{O}):=\prod_{\mathfrak{p}\in\Spec(\mathcal{O}_k)}G(\mathcal{O}_{\mathfrak{p}})\subset G(\mathbb{A}_k),\]
where 
\[\mathcal{O}_p=\mathcal{O}_{k_{\mathfrak{p}}},\]
then a maximal compact subgroup of $G(\mathbb{A}_k)$ is
\[\mathcal{K}=G(\mathcal{O})\times\overbrace{\text{SO}(2)\times\dots\times\text{SO}(2)}^{r_1}\times\overbrace{\text{SU}(2)\times\cdots\times\text{SU}(2)}^{r_2},\]
where $r_1$ (resp, $r_2$) denotes the number of real (resp. complex) places of $k$.
\newline $G(\mathbb{A}_k)$ operates by right translation on $L^2(G(k)\backslash G(\mathbb{A}_k)):=L^2(G(k)\backslash G(\mathbb{A}_k),1)$, the space of all functions $\phi$ on $G(k)\backslash G(\mathbb{A}_k)$ such that, for $a\in\mathcal{K}$ and $g\in G(\mathbb{A}_k)$:
\[\phi(ag)=\phi(g),\]
\[\int_{G(k)\mathcal{K}\backslash G(\mathbb{A}_k)}|\phi(g)|^2dg<\infty.\]
We will denote by $L_0^2(G(k)\backslash G(\mathbb{A}_k))$ the subspace of cuspidal elements, that is, those functions in $L^2(G(k)\backslash G(\mathbb{A}_k))$ such that for all proper $k$-parabolic subgroups $H$ of $G$, the integral
\[\int_{U(k)\backslash U(\mathbb{A}_k)}\phi(ug)du=0,\]
where $U$ is the unipotent radical of $H$.
\newline An admissible matrix coefficient \[f_{\pi}:G\rightarrow\mathbb{C}\] of the representation $\pi$ of $G(\mathbb{A}_k)$ on $L_0^2(G(k)\backslash G(\mathbb{A}_k))$ has the form:
\[f_{\pi}(g)=\int_{G(\mathcal{O})G(k)\backslash G(\mathbb{A}_k)}\phi(hg)\widetilde{\phi}(h)dh,\]
for cuspidal automorphic forms $\phi,\tilde{\phi}$, the sense of \cite[Chapter~II,~\S10]{ZFSA}. The module $m:G(\mathbb{A}_k)\rightarrow\mathbb{R}^{\times}_{+}$ is defined by:
\[m(g)=|g|:=|\det(g)|_{\mathbb{A}_k},\]
where the module on the right hand side is the usual adelic module.
\newline Let $\Phi$ be a Schwartz-Bruhat function on $M_n(\mathbb{A}_k)$. According to \cite[\S~13]{ZFSA}, we have an integral representation for the automorphic $L$-function
\[\int_{G(\mathbb{A}_F)}\Phi(g)f_{\pi}(g)|g|^sdg=L(\pi,s-\frac{1}{2}),\]
which can be analytically continued to $s\in\mathbb{C}$.
\newline For $x\in\mathbb{R}^{\times}_{+}$, define $G_x$ to be the set $\{g\in G(\mathbb{A}_k):|g|=x\}$. For $\Phi\in S(M_n(\mathbb{A}_k))$ and an admissible matrix coefficient $f_{\pi}$, the following integral converges absolutely
\[\mathscr{Z}(\Phi,f_{\pi}):\mathbb{R}_{+}^{\times}\rightarrow\mathbb{C}\]
\[x\mapsto\mathscr{Z}(\Phi,f_{\pi})(x):=\int_{G_x}\Phi(g)f_{\pi}(g)dg\]
Moreover, for any positive integer $N$, there exists a positive constant $C$ such that, for all $x\in\mathbb{R}^{\times}_{+}$, $|\mathscr{Z}(\phi,f_{\pi})(x)|\leq Cx^{-N}$ and, if $\widehat{\Phi}$ denotes the Fourier transform of $\Phi$ and $\check{f}_{\pi}(g)=f_{\pi}(g^{-1})$, then we have the functional equation
\[\mathscr{Z}(\Phi,f_{\pi})(x)=x^{-2}\mathscr{Z}(\widehat{\Phi},\check{f}_{\pi})(x^{-1}).\]
Let $S(\pi)$ denote the set of pairs $(\Phi,f_{\pi})$, where $\Phi\in S(M_n(\mathbb{A}_k))$ is a Schwartz-Bruhat function and $f_{\pi}$ is an admissible coefficient of $\pi$. The properties above verify that we have defined a map
\[\mathscr{Z}:S(\pi)\rightarrow\textbf{S}(\mathbb{R}_{+}^{\times})\]
\[(\Phi,f_{\pi})\mapsto\mathscr{Z}(\Phi,f_{\pi}).\]
We will denote the image of $\mathscr{Z}$ by $\mathcal{V}_{\pi}\subset\textbf{S}(\mathbb{R}^{\times}_{+})$.
\newline Let
\[w_0(x)=\frac{1}{2\pi i}\int_{(c)}\Gamma(C,s)^2A(C)^ss^4(s-2)^4(s-1)^{2}x^{-s}ds,\]
The following space gives a family of convolutors of the function $h_C(s)$ of definition~\ref{hC.definition}:
\[\mathcal{W}_{\pi}=w_0\ast\mathcal{V}_{\pi}\ast\mathcal{V}_{\pi}:=\overline{\text{Span}_{\mathbb{C}}\{w_0\ast v_1\ast v_2:v_i\in\mathcal{V}_{\pi}\}}.\]
More precisely,
\begin{theorem}\label{MPGL.theorem}
Let $C$ be a smooth projective curve with associated cuspidal automorphic representation $\pi$, and let $h_{C}$ be as in definition~\ref{hC.definition}, then
\[\mathcal{W}_{\pi}\subset\mathcal{T}(h_{C})^{\perp},\]
where $\mathcal{T}(h_{C})^{\perp}$ denotes the set of convolutors $\{g\in\textbf{S}(\mathbb{R}^{\times}_{+}):g\ast\tau=0,\forall\tau\in\mathcal{T}(h_{C})\}$.
\end{theorem}
\begin{proof}
Again, we have the expansion of $h_C$ in lemma~\ref{expansion.lemma}. The poles $\lambda$ are either the poles of shifted Dedekind zeta functions, which we can quantify, or the zeros of $\Lambda(C,s)$. Computing the convolution as in the theorem of the previous section, we are again lead to a series of zeros.
\end{proof}
In order to recover the result for Hecke characters one must induce the associated representation of $GL_1(\mathbb{A}_k)\times GL_1(\mathbb{A}_k)$ to $GL_{2}$. Dually, one can also prove that the set of translates to the boundary function are orthogonal the modified image of the zeta integrals.
\begin{theorem}
Let $C$ be a smooth projective curve over a number field $k$, and let $\pi$ be as above, then
\[\mathcal{T}(h_{C})\subset\mathcal{W}_{\pi}^{\perp}.\]
\end{theorem}
We will now use an adelic construction to prove a similar result to theorem~\ref{MPGL.theorem}. To do so, we will write $n=2g$ and introduce the following subset of the Schwartz-Bruhat space $S(M_n(\mathbb{A}_k))$:
\[S(M_n(\mathbb{A}_k))_0=\{\Phi\in S(M_n(\mathbb{A}_k)):\Phi(x)=\widehat{\Phi}(x)=0,\text{ for }x\in M(\mathbb{A}_k)-G(\mathbb{A}_k)\}.\]
For $\Phi\in S(M_n(\mathbb{A}_k))_0$, following Deitmar \cite{APHOFALF}, we define the functions
\[E(\Phi),\widehat{E}(\Phi):G(\mathbb{A}_k)\rightarrow\mathbb{C}\]
\[E(\Phi)(g)=|g|^{n/2}\sum_{\gamma\in M_n(\mathbb{A}_k)}\Phi(\gamma g)\]
\[\widehat{E}(\Phi)(g)=|g|^{n/2}\sum_{\gamma\in M_n(\mathbb{A}_k)}\widehat{\Phi}(g\gamma).\]
By \cite[theorem 3.1, lemma 3.5]{APHOFALF}, when $\phi_{\pi}=\otimes_v\phi_{\pi,v}$ is such that $\phi_{\pi,v}$ is a normalized class one vector for almost all places, we have that there is an entire function $F_{\Phi,\phi_{\pi}}$ of $s\in\mathbb{C}$ such that
\[\int_{G(k)\backslash G(\mathbb{A}_k)}E(\Phi)(g)\phi_{\pi}(g)|g|^{s-n/2}dg=L(\pi,s-\frac{1-n}{2})F_{\Phi,\phi_{\pi}}(s)\]
We will denote the analogue of $\mathscr{Z}$ above by $\mathcal{Z}$, which we construct following \cite[3.2]{TDAAACAROGL2}. For $\Phi\in S(M_n(\mathbb{A}_k))_0$, define
\[\mathcal{Z}(\Phi),\widehat{\mathcal{Z}}(\Phi)\in\textbf{S}(G(k)\backslash G(\mathbb{A}_k))=\bigcap_{\beta\in\mathbb{R}}|~|^{\beta}S(G(k)\backslash G(\mathbb{A}_k))\]
\[\mathcal{Z}(\Phi)(g)=\frac{E(\Phi)(g)}{|g|^{n/2}}=\sum_{\gamma\in M_n(\mathbb{A}_k)}\Phi(\gamma g)\]
\[\widehat{\mathcal{Z}}(\Phi)(g)=\frac{\widehat{E}(\Phi)(g)}{|g|^{n/2}}\sum_{\gamma\in M_n(\mathbb{A}_k)}\widehat{\Phi}(g\gamma).\]
The fact that these functions are in the strong Schwartz space follows from \cite[Proposition~3.1]{APHOFALF}. We have the module map
\[m:G(\mathbb{A}_k)\rightarrow\mathbb{R}^{\times}_{+}\]
\[m(g)=|\det(g)|_{\mathbb{A}_k},\]
which fits into the split exact sequence
\[1\rightarrow G(\mathbb{A}_k)_1\rightarrow G(\mathbb{A}_k)\rightarrow\mathbb{R}_{+}^{\times}\rightarrow1,\]
where $G(\mathbb{A}_k)_1$ denotes the elements of $G(\mathbb{A}_k)$ of module $1$, and the arrows are the natural ones. 
\newline With $\phi_{\pi}$ as above, define
\[\mathcal{W}_{\pi}=\overline{\text{Span}_{\mathbb{C}}\{(w_0\circ m)\ast(\mathcal{Z}(\phi_1)\cdot\phi_{\pi})\ast(\mathcal{Z}(\phi_2)\cdot\phi_{\pi}):\phi_i\in S(M_n(\mathbb{A}_k))\}},\]
where $\mathcal{Z}(\phi)\cdot\phi_{\pi}(x)=\mathcal{Z}(\phi)(x)\phi_{\pi}(x)$ and $\ast$ is the convolution via the right regular representation. The orthogonal complement of $\mathcal{W}_{\pi}$ with respect to convolution is denoted $\mathcal{W}_{\pi}^{\perp}$:
\[\mathcal{W}_{\pi}^{\perp}:=\{\eta\in\textbf{S}(G(k)\backslash G(\mathbb{A}_k))^{\ast}:w\ast\eta=0,\forall w\in\mathcal{W}_{\pi}\}.\]
For all $\eta\in\textbf{S}(G(k)\backslash G(\mathbb{A}_k))^{\ast}$, define
\[\mathcal{T}(\eta)=\overline{\text{Span}_{\mathbb{C}}\{R^{\ast}(g)\eta:g\in G(\mathbb{A}_k)\}},\]
where $R$ is the right regular representation of $G(\mathbb{A}_k)$ on $S(G(k)\backslash G(\mathbb{A}_k))$ and $R^{\ast}$ is its transpose with respect to the pairing of $S(G(k)\backslash G(\mathbb{A}_k))$ and $S(G(k)\backslash G(\mathbb{A}_k))^{\ast}$.
\newline We note that the following theorem generalizes \cite[Theorem~3.2]{TDAAACAROGL2} as well as correcting a small mistake.
\begin{theorem}
If $C$ is a smooth projective curve and $\mathcal{W}_{\pi}$ is as above, then
\[\mathcal{T}(h_{C}\circ m)\subset\mathcal{W}_{\pi}^{\perp}.\]
\end{theorem}
\begin{proof}
The proof proceeds as previously, via the expansion of $h_C(s)$ in \ref{expansion.lemma}. The poles $\lambda$ are again the zeros of $\Lambda(\pi,s-\frac{1}{2})$ or those of the Mellin transform of $w_0$.
\end{proof}
\section{Whittaker Expansions}\label{whittaker.section}
The mean-periodicity correspondence tells us that one may infer the meromorphic continuation and functional equation of $L$-functions of arithmetic curves $C$ from a convolution against what we have denoted $h_C$. In this section we will compare this convolution to those in the Rankin--Selberg method. More precisely, we will compare $h_C$ to an Eisenstein series.
\subsection{Aside - Fourier Series}
The ideas of this subsection are not strictly necessary in the reading of this paper, though may provide some helpful intuition of mean-periodic functions, especially for what follows.
\newline Each strong Schwartz function on $\mathbb{R}$ is a distribution of over-exponential decay, and each function of at most exponential growth on $\mathbb{R}$ is a weak tempered function. To be precise, there are  continuous injections
\[\mathscr{C}^{\infty}_{\text{exp}}(\mathbb{R})\hookrightarrow\textbf{S}(\mathbb{R})^{\ast},\]
\[\textbf{S}(\mathbb{R})\hookrightarrow\mathscr{C}^{\infty}_{\text{exp}}(\mathbb{R})^{\ast},\]
where $\mathscr{C}^{\infty}_{\text{exp}}(\mathbb{R})$ is the space of smooth functions on $\mathbb{R}$ with at most exponential growth at $\pm\infty$, that is
\[\forall n\in\mathbb{Z}_{+},\exists m\in\mathbb{Z}_{+},f^{(n)}(x)=O(\text{exp}(m|x|))\text{ as }x\rightarrow\pm\infty.\]
This space is an inductive limit of Fr\'{e}chet spaces $(F_m)_{m\geq1}$, where $F_m$ is the space of smooth functions on $\mathbb{R}$ such that for each positive integer $n$, $f^{(n)}(x)=O(\exp(m|x|))$ as $x\rightarrow\pm\infty$. There is a topology on each $F_m$ which is induced by the family of seminorms
\[||f||_{m,n}=\sup_{x\in\mathbb{R}}|f^{(n)}(x)\text{exp}(-m|x|)|.\]
The space $\mathscr{C}^{\infty}_{\text{poly}}(\mathbb{R}_{+}^{\times})$ of smooth functions from $\mathbb{R}^{\times}_{+}$ to $\mathbb{C}$ which have at most polynomial growth at $0^{+}$ and $+\infty$ is defined as those smooth functions $\mathbb{R}^{\times}_{+}$ such that
\[\forall n\in\mathbb{Z}_{+},\exists m\in\mathbb{Z},f^{(n)}(t)=O(t^m)\text{ as }x\rightarrow0^{+},+\infty,\]
and its topology is such that the bijection
\[\mathscr{C}^{\infty}_{\text{exp}}(\mathbb{R})\rightarrow\mathscr{C}^{\infty}_{\text{poly}}(\mathbb{R}_{+}^{\times})\]
\[f(t)\mapsto f(-\log(x)),\]
becomes a homeomorphism. Dual to these spaces, with respect to the weak $\ast$-topology, we have the spaces of distributions of over polynomial decay. We have continuous injections
\[\mathscr{C}^{\infty}_{\text{poly}}(\mathbb{R}^{\times_{+}})\hookrightarrow\textbf{S}(\mathbb{R}^{\times_{+}})^{\ast}\]
\[\textbf{S}(\mathbb{R}^{\times_{+}})\hookrightarrow\mathscr{C}^{\infty}_{\text{poly}}(\mathbb{R}^{\times_{+}})^{\ast}.\]
As proved in \cite{SSISTVSOF}, the advantage of the spaces $\mathscr{C}^{\infty}_{\text{exp}}(\mathbb{R})$ and $\mathscr{C}^{\infty}_{\text{poly}}(\mathbb{R}^{\times_{+}})$ and their duals is that they admit spectral synthesis and, subsequently, a theory of generalized Fourier series, as we now explain. 
\newline Let $G=\mathbb{R}$ (resp. $\mathbb{R}^{\times}_{+}$) and let $X$ be a complex vector space of functions on $G$. We will assume that there exists an open $\Omega\subset\mathbb{C}$ such that for all $P(T)\in\mathbb{C}[T]$ and any $\lambda\in\Omega$ the exponential polynomials:
\[\begin{cases}
P(t)e^{\lambda t},\text{ }G=\mathbb{R}, \\
x^{\lambda}P(\text{log}(x)),\text{ }G=\mathbb{R}^{\times}_{+},
\end{cases}\]
belong to $X$. We say that spectral synthesis holds in $X$ if, whenever $\mathcal{T}(f)\neq X$,
\[\mathcal{T}(f)=\begin{cases} \overline{\text{Span}_{\mathbb{C}}\{P(t)e^{\lambda t}\in\mathcal{T}(f),\lambda\in\Omega\}}, & G=\mathbb{R},\\ \overline{\text{Span}_{\mathbb{C}}\{x^{\lambda}P(\log(x))\in\mathcal{T}(f),\lambda\in\Omega\}}, & G=\mathbb{R}^{\times}_{+}.\end{cases}\] 
When spectral synthesis holds in $X$, $f$ is $X$-mean-periodic if and only if $f$ is a limit of a sum of exponential polynomials in $\mathcal{T}(f)$.
\subsection{Boundary Functions and Eisenstein Series}
Let $C$ be a smooth projective curve over a number field $k$ and consider the following product of completed zeta functions:
\[Z_C(s)=\xi(k,s)^2\xi(C,2s)^2,\]
where $\xi(C,s)$ denotes the (completed) Hasse--Weil zeta function of $C$. In this section we will assume that $L(C,s)$ admits meromorphic continuation and satisfies the expected functional equation. Consequently
\[Z_C(s)=Z_C(1-s).\] 
The boundary function $h_C:\mathbb{R}_{+}^{\times}\rightarrow\mathbb{C}$ associated to $Z_C$ is
\[h_C(x)=f_C(x)-\varepsilon x^{-1}f_C(x^{-1}),\]
where $f_C$ is the inverse Mellin transform
\[f_C(s)=\frac{1}{2\pi i}\int_{(c)}Z_C(s)x^{-s}ds.\]
We can express $\xi(C,s)$ as a ratio of completed $L$-functions
\[\xi(C,s)=\frac{\xi(k,s)\xi(k,s-1)}{\Lambda(C,s)},\]
so that, according to \cite[Theorem~4.7]{SRF}, $h_C(s)$ satisfies the convolution equation
\[v\ast h_C=0,\]
where $v$ is the inverse Mellin transform of the denominator
\[v(x)=\frac{1}{2\pi i}\int_{(c)}\Lambda(C,2s)P(s)x^{-s}dx,\]
where $P(x)$ is the following polynomial
\[P(s)=s(s-1)\]
Writing
\[Z_C(s)=\gamma(s)D(s),\]
where $\gamma(s)$ is specified below and  $D(s)$ is the Dirichlet series representing $\zeta(C,2s)^2$, which we will write
\[D(s)=\sum_{m\geq1}\frac{c_m}{m^s},\]
we have
\[h_C(x)=\sum c_m[\kappa(mx)-\varepsilon x^{-1}\kappa(mx^{-1})],\]
where $\kappa(x)$ is the inverse Mellin transform of $\gamma(s)$. $\gamma(s)$ is a product involving the square of the completed Dedekind zeta function $\xi(k,s)$, and so its inverse Mellin transform is related to Eisenstein series for $\GL_{2}(\mathbb{A}_k)$, and hence has a Whittaker expansion which we will compute below. This connection to Eisenstein series goes a small way to substantiating the comparison to the Rankin-Selberg method of the convolutions studied in the previous sections.
\newline For simplicity, let $C=E$ be an elliptic curve over $\mathbb{Q}$, then the gamma factor of $\zeta(C,s)$ is trivial and
\[\kappa(x)=\frac{1}{2\pi i}\int_{(c)}\xi(k,s)^2x^{-s}ds.\]
We will understand this inverse Mellin transform as an Eisenstein series on $\GL_2(\mathbb{A}_{\mathbb{Q}})$ and state the associated Whittaker expansion.
\newline For $w\in\mathbb{C}$ such that $\Re(w)>1$, define
\[\mathscr{E}(z,w)=\frac{1}{2}\sum_{(c,d)=1}\frac{y^w}{|cz+d|^{2w}},\]
where $z=x+iy$, $y>0$ and the sum is taken over coprime integers $c$ and $d$.
\newline According to \cite[Theorem~3.1.8]{AFLFGLR}, we have the Fourier-Whittaker expansion:
\begin{align}
\mathscr{E}(z,w) = & y^w +\phi(w)y^{1-w}\nonumber \\
  & +\frac{2^{1/2}\pi^{w-1/2}}{\Gamma(\mathbb{Q},w)\zeta(\mathbb{Q},2w)}\sum_{n\neq0}\sigma_{1-2w}(n)|n|^{w-1}\sqrt{2\pi|n|y}K_{w-1/2}(2\pi|n|y)e^{2\pi nx},\nonumber
\end{align}
where $\phi$ is the following function
\[\phi(s)=\sqrt{\pi}\frac{\Gamma(s-\frac{1}{2})}{\Gamma(s)}\frac{\zeta(\mathbb{Q},2s-1)}{\zeta(\mathbb{Q},2s)}.\]
Note that the modified Bessel functions appearing in this expansion are closely related to Whittaker functions. To be precise
\[K_v(z)=(\frac{\pi}{2z})^{1/2}W_{0,v}(2z).\]
The $L$-function of $\mathscr{E}(z,w)$ is defined as follows, for $\Re(s)>\min\{1/2-w,w-1/2\}$
\[\mathcal{L}(\mathscr{E},s)=\sum_{n=1}^{\infty}\sigma_{1-2w}(n)n^{w-s-1/2}.\]
As in \cite[\S3.14]{AFLFGLR}, one can deduce that this is simply a product of Riemann 
\[\mathcal{L}(\mathscr{E},s)=\zeta(\mathbb{Q},s+w-1/2)\zeta(\mathbb{Q},s-w-1/2).\]
We therefore have meromorphic continuation and a simple functional equation for $\mathcal{L}(\mathscr{E},s)$. We recover the following expansion for $h_C(x)$:
\[h_C(x)=\sum_{n\geq1}(\sum_{d|n}c_d\sigma_0(n/d))[K_0(2\pi ne^{-t})-e^tK_0(2\pi ne^t)].\]
Such expressions are far more difficult to obtain over arbitrary number fields, as well as for curves of arbitrary genus.
\subsection*{Acknowledgement}
The author was supported by the Heilbronn Institute for Mathematical Research. The author is grateful to Ivan Fesenko for much encouragement during the early days of this paper and to the anonymous referee for several helpful suggestions which have been included in the current version.
\bibliography{studygroup}

\begin{thebibliography}{10}

\bibitem{DRCACOC}
S.~Bloch.
\newblock De {R}ham cohomology and conductors of curves.
\newblock {\em Duke Math. J.}, 54:295--308, 1987.

\bibitem{APHOFALF}
A.~Deitmar.
\newblock A {P}olya-{H}ilbert operator for automorphic {L}-functions.
\newblock {\em Indag. Mathem.}, 12:157--175, 2001.

\bibitem{LFMP}
J.~Delsarte.
\newblock Les fonctions moyennes-periodiques.
\newblock {\em Journal de Math. Pures et Appl.}, 1:403--453, 1935.

\bibitem{SRF}
I.~Fesenko, G.~Ricotta, and M.~Suzuki.
\newblock {M}ean-periodicity and zeta functions.
\newblock {\em Ann. L'Inst. Fourier}, 12:1819--1887, 2012.

\bibitem{F4}
I.B. Fesenko.
\newblock Adelic approach to the zeta function of arithmetic schemes in
  dimension two.
\newblock {\em Moscow Math. J.}, 8:273 -- 317, 2008.

\bibitem{F3}
I.B. Fesenko.
\newblock Analysis on arithmetic schemes {I}{I}.
\newblock {\em J. K-theory}, 5:437--557, 2010.

\bibitem{ZFSA}
R.~Godement and H.~Jacquet.
\newblock {\em Zeta Functions of Simple Algebras}.
\newblock Springer-Verlang, 1972.

\bibitem{AFLFGLR}
D.~Goldfeld.
\newblock {\em Automorphic Forms and L-functions for the Group GL(n,R)},
  volume~1.
\newblock Cambridge Studies in Advanced Mathematics, 2006.

\bibitem{LOMPF}
J.~P Kahane.
\newblock {\em Lectures on Mean Periodic Functions}.
\newblock Tata Institute of Fundamental Research, 1959.

\bibitem{WSFT}
R.P. Langlands.
\newblock Where stands functoriality today?
\newblock {\em Proc. Sympos. Pure Math.}, 61:457–471, 1997.

\bibitem{ASIFTZOTRZF}
R.~Meyer.
\newblock {A} spectral interpretation for the zeros of the {R}iemann zeta
  function.
\newblock {\em Mathematisches Institut, Georg-August-Universit\"{a}t
  G\"{o}ttingen: Seminars Winter Term 2004/2005}, pages 117--137, 2005.

\bibitem{OAROTICGRTPAZOLF}
R.~Meyer.
\newblock On a representation of the idele class group related to {P}rimes and
  zeros of {L}-{F}unctions.
\newblock {\em Duke Mathe. J.}, 127:519--595, 2005.

\bibitem{MeZetaintegrals}
T.~D. Oliver.
\newblock {Z}eta functions on arithmetic surfaces.
\newblock arXiv:1311.6964[math.NT], 2013.

\bibitem{SSISTVSOF}
S.S. Platonov.
\newblock Spectral synthesis in some topological vector spaces of functions.
\newblock {\em St. Petersburg Math. J.}, 22:813--833, 2011.

\bibitem{OANCARAACODS}
A.~Selberg.
\newblock Old and new conjectures and results about a class of {D}irichlet
  series.
\newblock {\em Univ. Salerno}, Proceedings of the Amalfi Conference on Analytic
  Number Theory (Maiori, 1989):367--385, 1992.

\bibitem{FLDFZDVA}
J.P. Serre.
\newblock Facteurs locaux des fonctions zeta des varietes algebriques.
\newblock {\em Seminaire Delagne-Pisot-Poitou}, 19:1--15, 1969.

\bibitem{S}
J.H. Silverman.
\newblock {\em Advanced Topics in the Arithmetic of Elliptic Curves}.
\newblock Springer, 1994.

\bibitem{POCFAWAOES}
M.~Suzuki.
\newblock Positivity {O}f {C}ertain {F}unctions {A}ssociated {W}ith {A}nalysis
  {O}n {E}lliptic {S}urfaces.
\newblock {\em Journal of Number Theory}, 131:1770--1796, 2011.

\bibitem{TDAAACAROGL2}
M.~Suzuki.
\newblock Two dimensional adelic analysis and cuspidal automorphic
  representations of {G}{L}(2).
\newblock {\em Progress in Mathematics; Multiple Dirichlet Series, L-functions
  and Automorphic Forms}, 300:339--361, 2012.

\bibitem{FZED}
A.~Weil.
\newblock Fonction zeta et distributions.
\newblock {\em Seminaire Bourbaki 9}, Exp. No. 312:523--531, 1965/6.

\end{thebibliography}
\bibliographystyle{plain}
\end{document}